\documentclass[a4paper]{amsart}
\usepackage{amsfonts}
\usepackage{amssymb}
\usepackage{graphicx}    % standard LaTeX graphics tool
\usepackage{multicol}    % used for the two-column index

\usepackage{ifdraft}
\usepackage{todonotes}

\usepackage{amsmath}
\usepackage[hidelinks]{hyperref}
\usepackage{tcolorbox}
\usepackage{mathrsfs}
\usepackage{amsthm}
\usepackage{lastpage}
\usepackage{fancyhdr}
\usepackage{accents}
\usepackage{stmaryrd}
\usepackage{tabularx}
\usepackage{tikz}
\usepackage{tikz-cd}
\usepackage{array}
\usepackage{mathtools}
\usepackage{comment}
\usepackage{subcaption}
\usepackage{adjustbox}
\usepackage{quiver}
\usepackage{colortbl}
\usepackage{pifont}
\DeclareFontFamily{OT1}{pzc}{}
\DeclareFontShape{OT1}{pzc}{m}{it}{<-> s * [1.10] pzcmi7t}{}
\DeclareMathAlphabet{\mathpzc}{OT1}{pzc}{m}{it}
\newtheorem{thm}{Theorem}[section]
\newtheorem{prop}[thm]{Proposition}
\newtheorem{cor}[thm]{Corollary}
\newtheorem{lem}[thm]{Lemma}
\newtheorem{conj}[thm]{Conjecture}

\newtheorem*{thma}{Theorem \ref{truemain}}
\newtheorem*{thmb}{Theorem \ref{maindouble}}
\theoremstyle{definition}
\newtheorem{defin}[thm]{Definition}
\newtheorem{construction}[thm]{Construction}
\newtheorem{exm}[thm]{Example}

\newtheorem{note}[thm]{Note}
\newtheorem{notation}[thm]{Notation}

\def\Z{\mathbb{Z}}    %%%%%%%%% the set of integers
\def\K{\mathcal{K}}
\def\H{\tilde{H}}
\def\A{\mathcal{A}}

\def\ZZ{\mathcal{Z}}

\def\R{\mathbb{R}}

\def\Z{\mathbb{Z}}
\def\L{\mathcal{L}}

\def\ker{\text{ker }}
\def\im{\text{Im }}

\def\F{\mathbb{F}}
\def\rank{\text{rank }}

\def\J{{J^c}}

\usepackage{tcolorbox}
\newtcolorbox{donbox}{colback=violet!10!white,colframe=violet!75!black,fonttitle=\bfseries,title=Stanley:}
\newtcolorbox{gabbox}{colback=teal!10!white,colframe=teal!75!black,fonttitle=\bfseries,title=Valenzuela:}
\newtcolorbox{daibox}{colback=teal!10!white,colframe=teal!75!black,fonttitle=\bfseries,title=Kishimoto:}

\usepackage{etoolbox}
\makeatletter
\patchcmd{\@maketitle}
  {\ifx\@empty\@dedicatory}
  {\ifx\@empty\@date \else {\vskip3ex \centering\footnotesize\@date\par\vskip1ex}\fi
   \ifx\@empty\@dedicatory}
  {}{}
\patchcmd{\@adminfootnotes}
  {\ifx\@empty\@date\else \@footnotetext{\@setdate}\fi}
  {}{}{}
\makeatother

\newtheorem{innercustomgeneric}{\customgenericname}
\providecommand{\customgenericname}{}
\newcommand{\newcustomtheorem}[2]{%
  \newenvironment{#1}[1]
  {%
   \renewcommand\customgenericname{#2}%
   \renewcommand\theinnercustomgeneric{##1}%
   \innercustomgeneric
  }
  {\endinnercustomgeneric}
}\newcustomtheorem{customprop}{Proposition}
\newcustomtheorem{customthm}{\textbf{Theorem}}

\begin{document}

\title{Moment angle complexes and duality for tight manifolds}

\author{Daisuke Kishimoto, Donald Stanley, Carlos Gabriel Valenzuela Ruiz}

\date{\today}

\maketitle
\thispagestyle{empty}
\begin{abstract}
For a field $\F$ and a triangulated compact $\F$-orientable manifold $\K$, consider the homology of the associated Moment-Angle complex $H_*(\ZZ_\K;\F)$. We show the total homology rank $\beta(\ZZ_\K;\F)$ satisfies the inequality  $\beta(\ZZ_\K;\F)\geq 2^{m-1}(\beta(\K;\F)-2)+2$, with equality occurring exactly when the triangulation is $\F-tight$. Using Lefschetz duality, we introduce a short exact sequence of functors that, in turn, introduces a new duality theorem in Double Homology for tight manifold triangulations.

\end{abstract}

\maketitle
\setcounter{tocdepth}{1}
\tableofcontents

\section{Introduction}

    Throughout this work, all simplicial complexes will be finite and abstract (see Definition \ref{simpdef}) and their homology will be simplicial with coefficients in some field $\F$ unless otherwise specified. By triangulated manifold, we mean a simplicial complex $\K$ such that its geometric realization is a closed topological manifold.\\

    Let $\K$ be a simplicial complex with vertex set $[m]$ and for every $J\subseteq [m]$ denote by $\K_J$ the full subcomplex consisting of simplices in $\K$ that are subsets of  $J$. We say a complex $\K$ is $\F$-tight if for all $J\subseteq [m]$, the inclusion $\K_J\subseteq\K$ induces an injective map in homology with coefficients in $\F$. The idea of tight triangulations of manifolds was introduced in \cite{Kuhnel}, inspired by economical embeddings of spaces in an Euclidean space; Künhel and Lutz conjectured that a triangulation is strongly minimal, that is, it has the least number of faces in every dimension only if it is tight. They verified the conjecture for dimensions 1 and 2 in \cite{KunLutz}; dimension 3 was verified by Baguchi, Data, and Spreer in \cite{BagDatSpreer}. Little is known about higher dimensions.\\
    
    Intuitively a manifold triangulation being tight suggests that if we restrict to a subset of points the topological features of that subcomplex aren't lost when completing the complex. It is also worth noting that tightness is a global feature of a complex that is obtained by looking at each possible vertex restriction individually. This thought process is where we make the connection to \textit{moment-angle complexes} a central object in the field of Toric Topology (see \cite[Ch. 4]{ToricTop}) that encodes the combinatorial information of a simplicial complex $\K$ in $\ZZ_\K$ a subspace of the Euclidean Space $\mathbb{C}^m$ (see Section 2.3). As it turns out, the cohomology ring of $\ZZ_\K$ splits as a direct sum of the cohomology of every subcomplex of $\K$ (see Proposition \ref{Hochster}). Let us denote by $\beta(X;\F)$ the sum of the betti numbers of $X$ with coefficients in $\F$, the splitting mentioned before and Poincaré-Alexander-Lefschetz duality (Prop. \ref{pal}) lets us characterize tightness in our first result:

    \begin{thma}
        Let $\F$ be a field and $\K$ a triangulated closed $\F$-orientable $n$-manifold on $m$ vertices. Then
        \[\beta(\ZZ_\K;\F)\geq2^{m-1}(\beta(\K;\F)-2)+2\]
        with equality if and only if $\K$ is $\F$-tight.
    \end{thma}
    
    This result turns the problem of verifying injectivity of $2^m$ maps of the form $H_*(\K_J)\to H_*(\K)$ into that of computing the cohomology of the moment-angle complex $\K$.\\

    In \cite{Limonchenko_2023}, the authors defined the double homology functor $DH_{*,*}$ (originally denoted as $HH$); this construction was designed to solve a stability problem in the use of Tor-complexes for persistent homology \cite{Stab}. Apart from this application, this functor encodes non-trivial combinatorial properties of simplicial complexes in relatively small finite-dimensional algebras, raising interest in questions regarding its rank and internal structure 
    \cite{HAN2023108421}\cite{ruiz2024spheretriangulationsdoublehomology}\cite{zhang2024rankdoublecohomologymomentangle}\cite{caputi2025bridginguberhomologydoublehomology}\cite{STANLEY2025109319}. \\
    
    In Section 4, we introduce a generalisation of double homology to abelian poset functors of the form $F: 2^{[m]}\to \A$ and denote it as $H^*(F)$. This construction is reminiscent of Khovanov's work on the Jones' polynomial \cite{jonespoly}. We refer the reader to \cite[\S6]{chandler} for an exploration of this construction in more generality. For our second result, we introduce a sequence of maps arising from Lefschetz duality $\eta_q:H^*(H_q(\K_-))\to H^*(H^{n-q}(\K_{-^c}))$ (see Equation 5) for an $n$-dimensional triangulated manifold $\K$, which defines a new duality when $\K$ is tight.
    
    \begin{thmb}
        For every integer $q$, the map $\eta_q$ is an isomorphism whenever $\K$ is tight.
    \end{thmb}

\section{Preliminaries}
\subsection{Triangulations}
\begin{defin}\label{simpdef}
    A \textit{simplicial complex} $\K$ on an ordered set $S$ is a non-empty collection of subsets of $S$ such that $\tau\subseteq\sigma\in \K$ implies $\tau\in\K$. Throughout this work, the vertex set $S$ will be finite and therefore of the form $[m]:=\{1,\cdots, m\}$ for some integer $m$. 
\end{defin}
\begin{defin}
    Given a simplicial complex $\K$ on $[m]$ and $J\subseteq [m]$, the \textit{full subcomplex} $\K_J$ is the simplicial complex on $J$ given by 
    \[\K_J:=\{\sigma\in\K:\sigma\subseteq J\}.\]
    \end{defin}
\begin{notation}
    For a subset $J\subseteq[m]$, we will adopt the notation $J^c:=[m]-J$.
\end{notation}
\begin{defin}
    The \textit{geometric realization} of a simplicial complex $\K$ on $m$ is the topological space $|\K|\subseteq \R^m$ given by
    \[|\K|:=\bigcup_{\sigma\in\K} \Delta_\sigma\]
    where $\Delta_\sigma$ is the convex hull of $\{e_j:j\in \sigma\}$ and $(e_j)_{j\in [m]}$ is the standard basis of $\R^m$.
\end{defin}

\begin{note}
    For a simplicial complex $\K$, we denote by $H_*(\K)$ the simplicial homology of its geometric realization $|\K|$ for some field $\F$ unless otherwise specified. Similarly, $H^*(\K)$ denotes the simplicial cohomology.
\end{note}
\begin{defin}
    The join of two simplicial complexes $\K, \L$ is defined to be the complex of pairs $(\sigma,\tau)$ where $\sigma\in \K$ and $\tau\in \L$.
\end{defin}
\begin{notation}
    For a fild $\F$, we will denote the Betti numbers of a complex $\K$ as $\beta_j(\K;\F):=\text{dim }H_j(\K;\F)$ and their sum as $\beta(\K;\F):=\sum\limits_{j\in\Z}\beta_j(\K;\F)$. We adapt the notation to $\Tilde{\beta}$ when taking the dimension of the reduced homology instead. We adapt the notation when talking about cohomology by turning the subscript into an upperscript.
\end{notation}
\begin{defin}
     We say a space $M$ is a topological $n$-manifold if it's locally homeomorphic to $\R^n$. We say a simplicial complex $\K$ is a \textit{triangulated manifold} if its geometric realization is homeomorphic to some closed (topological) manifold. 
\end{defin}
\subsection{Duality}

 \begin{defin}
     For a field $\F$, we say a closed $n$-manifold $M$ is $\F$-orientable if $H_n(M;\F)\neq 0$.
 \end{defin}
 
Topological manifolds satisfy several duality theorems, which makes their (co)homology particularly interesting to study. A proof of the following theorem can be found in Theorem 8.3 of \cite{bredon} by setting $L=\emptyset$. 
\begin{prop}[Poincaré-Alexander-Lefschetz duality]\label{pal}
     If $K$ is a compact locally contractible subspace of a closed $\F$-orientable $n$-manifold $M$ then there is a natural isomorphism \[H_i(M,M-K;\F)\to H^{n-i}(K;\F).\]
\end{prop}

To effectively use this in the world of abstract simplicial complexes, we will need the following lemma:
\begin{lem}\label{simp}
    Let $\K$ be a simplicial complex on $[m]$ and $J\subseteq [m]$. There is a strong deformation retraction $|\K|-|\K_J|\simeq |\K_{{J^c}}|$.
\end{lem}
\begin{proof}
    Let $e_1,\ldots, e_m$ be the standard basis of $\R^m$. First we consider the case $\K=2^{[m]}$; consider the homotopy $H:(\Delta_{[m]}-\Delta_{J})\times I\to \Delta_{[m]}-\Delta_J$ given by 
    \[H(x,t)=(1-t)x+t\left(\frac{\sum\limits_{j\in {J^c}}x_je_j}{\sum\limits_{j\in {J^c}}x_j}\right)\]
    this is well defined as, since $x\notin\Delta_J$ there's some $i\in {J^c}$ such that $x\cdot e_i\neq 0$. This is clearly a deformation retraction into $\Delta_{\J}$ as $H(x,1)\in \Delta({J^c})$ and whenever $x\in \Delta ({J^c})$, the second summand in the formula is just $tx$ and so $H(x,t)=(1-t)x+tx=x$.\\

    Notice that whenever a point $x$ is in a simplex $\sigma$ that intersects $\Delta_{J^c}$, the homotopy slides it towards its maximal face contained in $\Delta_{J^c}$ while staying inside of $\sigma$. If we consider a simplicial complex $\K$ and $x\in \K-\K_J-\K_{J^c}$, then $x$ is as in the previous observation, meaning that $H$ is well defined when intersecting with $\K$, showing this is indeed a strong deformation retract.
\end{proof}

\begin{prop}\label{les}
    Let $\K$ be a triangulated closed $\F$-orientable $n$-manifold on $[m]$ and $J\subseteq [m]$, then there is a long exact sequence
    \[\begin{tikzcd}
	{H_q(\K_J)} & {H_q(\K)} & {H^{n-q}(\K_{J^c})} & {H_{q-1}(\K_J),}
	\arrow["{i_q(J)}", from=1-1, to=1-2]
	\arrow["{p_q(J)}", from=1-2, to=1-3]
	\arrow["{\delta_q(J)}", from=1-3, to=1-4]
\end{tikzcd}\]
    where $i_q(J)$ is the map induced by the inclusion $\K_J\xhookrightarrow{}\K$. Further, the sequence is natural with respect to inclusions $J\subseteq L\subseteq [m]$.
\end{prop}
\begin{proof}
    The exactness of this sequence follows from considering the long exact sequence of the pair $(\K,\K_J)$ and applying Proposition \ref{pal} and Lemma $\ref{simp}$. Naturality follows from the naturality of the long exact sequence of the pair and the naturality of the Poincaré-Alexander-Lefschetz duality.
    \end{proof}

\subsection{Tightness} In \cite{KunLutz}, the authors show the absolute curvature of an immersion of a manifold $M$ into an Euclidean space is minimal when the immersion is tight, that is, for every closed half space H of the Euclidean space, the inclusion $M\cap H\xhookrightarrow{}M$ is injective in homology.\\

Given a triangulated manifold $\K$ on $[m]$, we can consider the geometric realization included into $\R^m$; this gives us a combinatorial characterisation of tightness for simplicial complexes.
\begin{defin}\label{tightdef}
    For a field $\F$, a simplicial complex $\K$ is $\F$-tight if for every $J\subseteq [m]$ the inclusion $\K_J\subseteq \K$ induces an injective map $H_*(\K_J)\to H_*(\K)$. We say a complex is tight if it is $\F$-tight for some field $\F$.
\end{defin}

\begin{prop}
    A triangulation $\K$ of a space $X$ on $[m]$ is $\F$-tight if and only if the embedding $X\xrightarrow{\cong}|\K|\xhookrightarrow{} \R^m$ is $\F$-tight.
\end{prop}
\begin{proof}
    Let $e_1,\ldots, e_m$ be the standard basis of $\R^m$. For each half plane $h$ of $\R^m$ define $J_h:=\{j\in [m]:e_j\in h\}$; there's a deformation retraction $|\K|\cap h\simeq |\K_{J_h}|$ constructed analogously to the one we used in the proof of Lemma \ref{simp} by projecting to $J_h$ instead of $J^c$.      
\end{proof}

One interest in tight triangulations of manifolds arises from the well-known Kühnel-Lutz conjecture; we refer the reader to \cite{KunLutz} for more details:

\begin{conj}
    The triangulation of an n-manifold is tight if and only if it's strongly minimal.
\end{conj}
Where by \textit{strongly minimal} we mean its number of faces is minimal in every dimension. This conjecture was shown to be true for $n\leq 3$ in \cite{BagDatSpreer}.\\

Another reason for our interest in tight complexes is that it was recently shown in \cite{iriye2023tightcomplexesgolod} that the class of tight manifold triangulation is precisely the class of manifolds $\K$ such that the Massey products of $\ZZ_\K$ are all trivial.

\subsection{Moment-Angle Complexes}

\begin{construction}
    Consider a simplicial complex $\K$ on $[m]$, for every $J\subseteq [m]$ consider the following topological space
    \[(D^2,S^1)^J:=\prod_{j\in J} Y_{J;j}\subseteq (D^2)^m\hspace{1cm}\text{where}\hspace{1cm} Y_{J;j}=\left\{\begin{array}{cl}
        D^2 & \text{if }j\in J \\
         S^1& \text{else.} 
    \end{array}\right.\]
    The moment-angle complex is then defined as $\ZZ_\K:=\bigcup\limits_{\sigma\in\K}(D^2,S^1)^\sigma$. We refer the reader to \cite[\S4]{ToricTop} for more details.
\end{construction}

\begin{defin}
    For a simplicial complex $\K$ on $[m]$, we define the \textit{Stanley-Reisner} ring of $\K$ as $\F[\K]:=\F[v_1,\ldots, v_m]/I_\K$ where $I_\K:=\left\langle{\prod\limits_{i\in \sigma}v_i:\sigma\notin\K}\right\rangle$
\end{defin}

\begin{prop}[{\cite{Hochster}}{\cite[\S3.2]{ToricTop}}]\label{Hochster}
     Let $\K$ be a simplicial complex on $[m]$ and $R$ a commutative ring. There are isomorphisms of bigraded algebras
     \begin{align*}
         H^* (\Lambda[u_1,\ldots,u_m]\otimes \F[\K],d)
         &\cong\bigoplus_{J\subseteq [m]}\H^*(\K_J)
     \end{align*}
     where bideg $u_i=(-1,2)$, bideg $v_i=(0,2)$, and the differential $d$ is given by $du_i=v_i$ and $dv_i=0$, and $[f]\in \H^{q-1}(\K_J)$ has bidegree $(q-|J|,2|J|)$, with the multiplication of the latter being 
     
    \[\mu_{J,L}:\H^*(\K_{J})\otimes H^*(\K_L)\xrightarrow[]{\cong}\H^*(\K_J*\K_L)\to \H^*(\K_{J\sqcup L})\] for disjoint subsets $J,L$, where the last map is induced by the inclusion  $\K_{J\sqcup L}\xhookrightarrow{}\K_J*\K_L$.
 \end{prop}
 \begin{prop}[\cite{ToricTop}]
        Let $\K$ be a simplicial complex on $[m]$. Then, for a commutative ring $R$, there's an algebra isomorphism
        \[H^*(\ZZ_\K)\cong H^*(\Lambda[u_1,\ldots,u_m]\otimes \F[\K])\]
 \end{prop}
 \noindent This induces a bigrading on $H^*(\ZZ_\K)$, in particular 
\[H^p(\K_J)\cong \bigoplus_{-k+2l=p}H^{-k,2l}(\ZZ_\K)\;\;\;\text{where }\;\;H^{-k,2l}(\ZZ_\K)\cong \bigoplus_{\substack{J\subseteq [m]\\|J|=l}}\H^{l-k-1}(\K_J).\]

There's a dual decomposition in homology by taking reduced homology of the full subcomplexes instead. We refer to both of these as the \textit{Hochster's decomposition} of $\ZZ_\K$.
\subsection{Poset cohomology}

\begin{notation}
    Given a poset $(P,\leq)$, we regard $P$ as a category with a unique map $x\to y$ whenever $x\leq y$. We denote by $2^{[m]}$ the poset of subsets of $[m]$ with the order induced by set inclusion.
\end{notation}

\begin{construction}
Consider the poset category $2^{[m]}$ for some integer $m$. Let $\mathcal{A}$ be any abelian category and $F:2^{[m]}\to \A$ a functor, we can construct a cochain complex $C^*(F)$ as follows:
    \begin{itemize}
        \item  $C^l(F)=\bigoplus\limits_{\substack{J\subseteq[m]\\|J|=l}}F(J)$
        \item $d^l:C^l(F)\to C^{l+1}(F)$ is given by \[d^l=\sum\limits_{\substack{J\subseteq[m]\\|J|=l}}\sum\limits_{x\notin J}{\varepsilon(J,x)}F(J\to J\cup\{x\})\] where $\varepsilon(J,x)=(-1)^{|\{j\in J: \; j<x\}|}$.
        \item For a natural transformation $\eta:F\to G$, we set $C^*(\eta)$ to be induced by the $\eta(J)$ for each $J\subseteq[m]$.
    \end{itemize}
    This lets us define the cohomology of such functors as $H^l(F):=H^l(C^*(F))$.
\end{construction}
\begin{exm}
    Let $\K$ be a simplicial complex and $F_\K$ be the functor defined as follows:
    \begin{itemize}
        \item $F_\K(\sigma)=\left\{\begin{array}{cl}
        \Z f_\sigma &  \text{if $\sigma\in\K$}\\
        0 & \text{else.}
    \end{array}\right.$
        \item Whenever $\sigma\in \K$ and $x\in \sigma$ \[F(\sigma\setminus\{x\}\subseteq\sigma)(f_\tau)=\left\{\begin{array}{cl}
            f_\sigma & \text{if }\tau=\sigma\setminus\{x\} \\
            0 & \text{else.}
        \end{array}\right.\]
    \end{itemize}
    This gives us the reduced simplicial cochain complex of $\K$ with a shift: ~$C^*(F)\cong\Tilde{C}^{*-1}(\K;\Z)$.
\end{exm}
\begin{prop}[{\cite[Prop. 2.7]{ruiz2025doubleuberposethomology}}]
   Let $m$ be an integer and $\A$ an abelian category. Let us denote by $dg\A$ the category of chain complexes on $\A$. The construction above defines an exact functor $C^*:\text{Fun}\left(2^{[m]},\A\right)\to dg\A$. 
\end{prop}

This construction provides a generalization for double homology, a new invariant for $\K$ obtained by endowing $\H_*(\ZZ_\K)$ with a differential, giving us a bigraded chain complex $CH_{*,*}$. The homology of this complex is denoted as $DH_{*,*}(\ZZ_\K)$ (originally denoted as $HH_{*,*}(\ZZ_\K)$). One can define \textit{double cohomology} in a dual manner, we denote the cochain by $CH^{*,*}(\ZZ_\K)$ and the double cohomology by $DH^{*,*}(\ZZ_\K)$. We refer the reader to \cite{Limonchenko_2023} for the explicit construction and more details.
\begin{notation}
    Given a simplicial complex $\K$, we denote the functor $\K_-$ from the category $2^{[m]}$ to the category of finite simplicial complexes, mapping each subset $J$ to the complex $\K_J$.
\end{notation}
\begin{exm}
    For $k\in \Z$, $CH_{-k,2l}(\ZZ_\K)=C^l(\H_{l-k-1}(\K_-))$ up to a sign in the differential.
\end{exm}
\begin{exm}
    For $k\in \Z$, $CH^{-k,2l}(\ZZ_K)=C^{m-l}(\H^{l-k-1}(\K_{-^c}))$ up to a sign in the differential.
\end{exm}
A proof of the following proposition can be found in \cite[\S3.4]{jonespoly} as Proposition 4 in the context of commuting cubes:
\begin{prop}
    Let $F:2^{[m]}\to \A$. If there is $x\in[m]$ such that for every $J\subseteq [m]\setminus\{x\}$, $F(J\subseteq J\cup\{x\})$ is an isomorphism then $H^*(F)=0$.
\end{prop}
A simple example of such a functor is the diagonal functor $\Delta(X):2^{[m]}\to \A$ for $X\in \A$ mapping every subset to $X$ and every map to the identity.
\begin{cor}\label{acyclic}
    For any $X\in\A$, the diagonal functor $\Delta(X)$ is acyclic.
\end{cor}
\section{A characterisation for tightness}
In this section we prove our first theorem: a tight lower bound for the total rank of the moment-angle complex associated to a manifold which is achieved when the triangulation is tight.

\begin{lem}\label{main}
    Let $\F$ be a field and let $\K$ be a triangulated closed $\F$-orientable $n$-manifold on $[m]$. For every $J\subseteq [m]$ and $q\in \Z$ we have that 
    \[\beta_q(\K;\F)\leq \beta_q(\K_{J};\F)+ \beta_{n-q}(\K_{{J^c}};\F)\]
    with equality for every $J\subseteq [m]$ and for every $q\in\Z$ if and only if $\K$ is $\F$-tight. 
\end{lem}
\begin{proof}
    From Proposition \ref{les} we have that 
    \begin{align*}
        \beta_q(\K;\F)&=\text{null }p_q(J) +\rank p_q(J)\\
                   &=\rank i_q(J)+\text{null } \delta_q(J)\\
                   &=\left(\beta_q(\K_J)-\text{null }i_q(J)\right)+\left(\beta_{n-q}(\K_\J)-\rank i_q\right)\\
                   &=\beta_q(\K_J;\F)+\beta_{n-q}(\K_\J;\F)-(\text{null }i_q(J)+\text{null }i_{q-1}(J))\\
                   &\leq \beta_q(\K_J;\F)+\beta_{n-q}(\K_{\J};\F)
    \end{align*} 
    Equality happens if and only if $\text{null }i_q(J)+\text{null }i_{q-1}(J)=0$ or equivalently, $\text{null }i_q(J)=\text{null }i_{q-1}(J)=0$, if we impose this condition for every $q$ and $J$ we get that this is equivalent to the definition of tightness.
\end{proof}
\begin{thm}\label{truemain}
    Let $\F$ be a field and $\K$ a triangulated closed orientable $n$-manifold on $m$ vertices. Then
    \[\beta(\ZZ_\K;\F)\geq2^{m-1}(\beta(\K;\F)-2)+2\]
    with equality if and only if $\K$ is $\F$-tight.
\end{thm}
\begin{proof}
    Fix a subset $J$ and add the inequality in Lemma 3.1 over all indices $q$
    \[\beta(\K)=\sum_q\beta_q(\K)\leq \sum_q\beta_q(\K_J)+\sum_q\beta_{n-q}(\K_{\J})=\beta(\K_J)+\beta(\K_\J).\]
    
    If we add over all subsets of $[m]$ we get that
    \begin{equation}\label{sumoveralJ}
        2^m\beta(\K)\leq\sum_{J\subseteq[m]}\beta(\K_J)+\sum\beta(\K_\J)=2\sum_{J\subseteq[m]}\beta(\K_J).
    \end{equation}

    \begin{align*}
    \beta(\ZZ_\K)&=\sum_{J\subseteq[m]}\tilde{\beta}(\K_J)\\
        	       &=\sum_{\substack{J\subseteq[m]\\q>0}}\beta_q(\K_J)+\sum_{\emptyset\neq J\subseteq[m]}(\beta_0(\K_J)-1)+\tilde{\beta}_{-1}(\K_\emptyset)\\
                   &=\sum_{J\subseteq[m]}\beta(\K_J)-(2^{m}-1)+1\\
                   \beta(\ZZ_\K)+2^m-2&=\sum_{J\subseteq[m]}\beta(\K_J).    
    \end{align*}
    Putting this together with Inequality (\ref{sumoveralJ}) gives us that 
    \[2^{m-1}\beta(\K)\leq\sum_{J\subseteq[m]}\beta(\K_J)=\beta(\ZZ_\K)+2^m-2.\]
    or equivalently
    \[\beta(\ZZ_\K)\geq 2^{m-1}(\beta(\K)-2)+2.\]
    The equality condition is inherited from the condition in Lemma \ref{main}.
\end{proof}
\begin{exm}
    Consider the six vertex minimal triangulation of $\R P^2$. The Betti numbers of its moment-angle complex are\\
    \begin{minipage}{0.6\textwidth}
    \centering
    \begin{tikzpicture}[x=0.75pt,y=0.75pt,yscale=-0.6,xscale=0.6]
%uncomment if require: \path (0,300); %set diagram left start at 0, and has height of 300
%Shape: Triangle [id:dp8479688764489004] 
\draw  [draw opacity=0][fill={rgb, 255:red, 165; green, 242; blue, 224 }  ,fill opacity=0.58 ] (328.59,259.21) -- (421.15,205.77) -- (236.02,205.77) -- cycle ;
%Shape: Rectangle [id:dp5653748284957785] 
\draw  [draw opacity=0][fill={rgb, 255:red, 165; green, 242; blue, 224 }  ,fill opacity=0.58 ] (236.02,98.89) -- (421.15,98.89) -- (421.15,205.77) -- (236.02,205.77) -- cycle ;
%Shape: Triangle [id:dp07991422788361158] 
\draw  [draw opacity=0][fill={rgb, 255:red, 165; green, 242; blue, 224 }  ,fill opacity=0.58 ] (328.59,45.44) -- (421.15,98.89) -- (236.02,98.89) -- cycle ;
%Straight Lines [id:da7536316361382323] 
\draw [line width=1.5]    (421.15,98.89) -- (328.59,104.89) ;
%Straight Lines [id:da748508395568277] 
\draw [line width=1.5]    (236.02,98.89) -- (328.59,104.89) ;
%Shape: Boxed Line [id:dp4568724762596279] 
\draw [line width=1.5]    (421.15,98.89) -- (328.59,45.44) ;
%Shape: Boxed Line [id:dp8885587753239197] 
\draw [line width=1.5]    (328.59,45.44) -- (236.02,98.89) ;
%Shape: Boxed Line [id:dp4585856613350665] 
\draw [line width=1.5]    (328.59,104.89) -- (328.59,45.44) ;
%Straight Lines [id:da5519091254218191] 
\draw [line width=1.5]    (287.5,176.05) -- (369.67,176.05) ;
%Shape: Boxed Line [id:dp7774011573314006] 
\draw [line width=1.5]    (328.59,259.21) -- (236.02,205.77) ;
%Shape: Boxed Line [id:dp4613772043768256] 
\draw [line width=1.5]    (421.15,205.77) -- (328.59,259.21) ;
%Straight Lines [id:da42912716533247686] 
\draw [line width=1.5]    (328.59,259.21) -- (354.31,207.12) -- (369.67,176.05) ;
%Straight Lines [id:da8819867824687955] 
\draw [line width=1.5]    (328.59,259.21) -- (287.5,176.05) ;
%Shape: Boxed Line [id:dp7349503812082099] 
\draw [line width=1.5]    (421.15,205.77) -- (421.15,98.89) ;
%Shape: Boxed Line [id:dp2716505533111577] 
\draw [line width=1.5]    (236.02,205.77) -- (236.02,98.89) ;
%Straight Lines [id:da5105930555564036] 
\draw [line width=1.5]    (328.59,104.89) -- (369.67,176.05) ;
%Straight Lines [id:da45376814183343184] 
\draw [line width=1.5]    (328.59,104.89) -- (287.5,176.05) ;
%Straight Lines [id:da06624525955939187] 
\draw [line width=1.5]    (421.15,98.89) -- (369.67,176.05) ;
%Straight Lines [id:da5259627017286019] 
\draw [line width=1.5]    (236.02,98.89) -- (287.5,176.05) ;
%Shape: Ellipse [id:dp6955576024263941] 
\draw  [fill={rgb, 255:red, 80; green, 227; blue, 194 }  ,fill opacity=1 ] (414.56,98.89) .. controls (414.56,95.25) and (417.51,92.3) .. (421.15,92.3) .. controls (424.79,92.3) and (427.74,95.25) .. (427.74,98.89) .. controls (427.74,102.52) and (424.79,105.47) .. (421.15,105.47) .. controls (417.51,105.47) and (414.56,102.52) .. (414.56,98.89) -- cycle ;
%Shape: Ellipse [id:dp9364258513279022] 
\draw  [fill={rgb, 255:red, 80; green, 227; blue, 194 }  ,fill opacity=1 ] (322,259.21) .. controls (322,255.57) and (324.95,252.62) .. (328.59,252.62) .. controls (332.22,252.62) and (335.17,255.57) .. (335.17,259.21) .. controls (335.17,262.85) and (332.22,265.8) .. (328.59,265.8) .. controls (324.95,265.8) and (322,262.85) .. (322,259.21) -- cycle ;
%Shape: Ellipse [id:dp8409416505641167] 
\draw  [fill={rgb, 255:red, 80; green, 227; blue, 194 }  ,fill opacity=1 ] (229.43,98.89) .. controls (229.43,95.25) and (232.38,92.3) .. (236.02,92.3) .. controls (239.66,92.3) and (242.61,95.25) .. (242.61,98.89) .. controls (242.61,102.52) and (239.66,105.47) .. (236.02,105.47) .. controls (232.38,105.47) and (229.43,102.52) .. (229.43,98.89) -- cycle ;
%Shape: Ellipse [id:dp30845781803257544] 
\draw  [fill={rgb, 255:red, 80; green, 227; blue, 194 }  ,fill opacity=1 ] (322,45.44) .. controls (322,41.81) and (324.95,38.86) .. (328.59,38.86) .. controls (332.22,38.86) and (335.17,41.81) .. (335.17,45.44) .. controls (335.17,49.08) and (332.22,52.03) .. (328.59,52.03) .. controls (324.95,52.03) and (322,49.08) .. (322,45.44) -- cycle ;
%Shape: Boxed Line [id:dp9240590507509217] 
\draw [line width=1.5]    (287.5,176.05) -- (236.02,205.77) ;
%Shape: Boxed Line [id:dp5920946731589114] 
\draw [line width=1.5]    (421.15,205.77) -- (369.67,176.05) ;
%Shape: Ellipse [id:dp5820675972697273] 
\draw  [fill={rgb, 255:red, 80; green, 227; blue, 194 }  ,fill opacity=1 ] (280.91,176.05) .. controls (280.91,172.41) and (283.86,169.46) .. (287.5,169.46) .. controls (291.14,169.46) and (294.09,172.41) .. (294.09,176.05) .. controls (294.09,179.68) and (291.14,182.63) .. (287.5,182.63) .. controls (283.86,182.63) and (280.91,179.68) .. (280.91,176.05) -- cycle ;
%Shape: Ellipse [id:dp4977215643693782] 
\draw  [fill={rgb, 255:red, 80; green, 227; blue, 194 }  ,fill opacity=1 ] (322,104.89) .. controls (322,101.25) and (324.95,98.3) .. (328.59,98.3) .. controls (332.22,98.3) and (335.17,101.25) .. (335.17,104.89) .. controls (335.17,108.53) and (332.22,111.48) .. (328.59,111.48) .. controls (324.95,111.48) and (322,108.53) .. (322,104.89) -- cycle ;
%Shape: Ellipse [id:dp8660612362431019] 
\draw  [fill={rgb, 255:red, 80; green, 227; blue, 194 }  ,fill opacity=1 ] (363.08,176.05) .. controls (363.08,172.41) and (366.03,169.46) .. (369.67,169.46) .. controls (373.31,169.46) and (376.26,172.41) .. (376.26,176.05) .. controls (376.26,179.68) and (373.31,182.63) .. (369.67,182.63) .. controls (366.03,182.63) and (363.08,179.68) .. (363.08,176.05) -- cycle ;
%Shape: Ellipse [id:dp4084990867281686] 
\draw  [fill={rgb, 255:red, 80; green, 227; blue, 194 }  ,fill opacity=1 ] (229.43,205.77) .. controls (229.43,202.13) and (232.38,199.18) .. (236.02,199.18) .. controls (239.66,199.18) and (242.61,202.13) .. (242.61,205.77) .. controls (242.61,209.41) and (239.66,212.36) .. (236.02,212.36) .. controls (232.38,212.36) and (229.43,209.41) .. (229.43,205.77) -- cycle ;
%Shape: Ellipse [id:dp6420951160235122] 
\draw  [fill={rgb, 255:red, 80; green, 227; blue, 194 }  ,fill opacity=1 ] (414.56,205.77) .. controls (414.56,202.13) and (417.51,199.18) .. (421.15,199.18) .. controls (424.79,199.18) and (427.74,202.13) .. (427.74,205.77) .. controls (427.74,209.41) and (424.79,212.36) .. (421.15,212.36) .. controls (417.51,212.36) and (414.56,209.41) .. (414.56,205.77) -- cycle ;

% Text Node
\draw (328,20.4) node [anchor=north west][inner sep=0.75pt]    {$1$};
% Text Node
\draw (325,269.4) node [anchor=north west][inner sep=0.75pt]    {$1$};
% Text Node
\draw (218,81.4) node [anchor=north west][inner sep=0.75pt]    {$2$};
% Text Node
\draw (429.74,209.17) node [anchor=north west][inner sep=0.75pt]    {$2$};
% Text Node
\draw (218,208.4) node [anchor=north west][inner sep=0.75pt]    {$3$};
% Text Node
\draw (428,81.4) node [anchor=north west][inner sep=0.75pt]    {$3$};
% Text Node
\draw (282,146.4) node [anchor=north west][inner sep=0.75pt]    {$4$};
% Text Node
\draw (337,82.4) node [anchor=north west][inner sep=0.75pt]    {$5$};
% Text Node
\draw (347,178.4) node [anchor=north west][inner sep=0.75pt]    {$6$};

\end{tikzpicture}

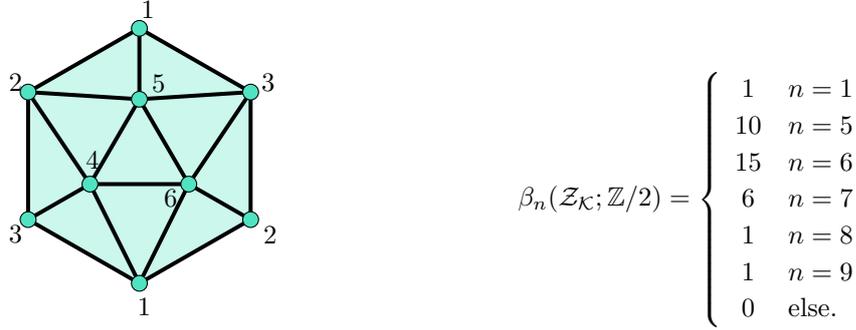
\captionof{figure}{Minimal triangulation of $\R P^2$}
    \end{minipage}\begin{minipage}{0.3\textwidth}
     \[\beta_n(\ZZ_\K;\Z/2)=\left\{\begin{array}{cl}
        1 & n=1  \\
        10 & n=5\\
        15 & n=6\\
        6 & n=7\\
        1 & n=8\\
        1 & n=9\\
        0 & \text{else.}
  \end{array}\right.\]
  \end{minipage}\\
  \phantom{}\\
  Notice that this satisfies the equality in the previous theorem since $\beta(\R P^2;\Z/2)=3$ and \[\beta(\ZZ_\K)=34=2^5(3-2)+2=2^{6-1}(\beta(\K)-2)+2,\]
  verifying this way the complex is tight. The above equation replaces the need to verify injectiveness of $2^6$ maps to show tightness of the complex.
\end{exm}

\section{A duality for tight complexes}
    It is of particular interest to understand the topology of $\ZZ_\K$ when $\K$ is a manifold. However, many of the topological properties of $\K$ as a manifold translate poorly to those of $\ZZ_\K$, in particular, $H^*(\ZZ_\K)$ will not be a Poincaré duality algebra in general for a manifold $\K$, in fact, $\ZZ_\K$ satisfies Poincaré duality if and only if $\K$ is a $\text{Gorensetein}^*$ complex \cite[Theorem 3.4.5]{cohen-mac-rings}.\\

    In this section, we present our second theorem: a duality map in double homology that appears in the special case of $\K$ being a tight triangulation.\\
    
    For every $q\in \Z$, the collection of maps $i_q(J)$, $p_q(J)$ and $\delta_q(J)$ from Proposition \ref{main} assemble into natural transformations $i_q$, $p_q$ $\delta_q$ which fit in the following exact sequence of functors
    
    \[H_q(\K_-)\xrightarrow{i_q}\Delta(H_q(\K))\xrightarrow{p_q}H^{n-q}(\K_{-^c}).\]

    This gives us the following three short exact sequences
    \begin{align}
        0\to \ker i_q\xhookrightarrow{} H_q&(\K_-)\xrightarrow{i_q} \im i_q\to 0,\\
        0\to \im i_q\xhookrightarrow{} \Delta_q&(\K)\xrightarrow{p_q} \im p_q\to 0,\\
        0\to \im p_q\xhookrightarrow{} H^{n-q}&(\K_{-^c})\xrightarrow{\delta_q} \ker i_{q-1}\to 0.
    \end{align}
Since $\Delta(H_q(\K))$ is acyclic, the connecting homomorphisms $\partial_{q}:H^{*}(\im p_q)\to H^{*+1}(\im i_q)$ are isomorphisms. This lets us define for every integer $q$ the following map in cohomology: \begin{equation}{\label{Etaq}}
    \eta_q:=\partial_q^{-1}i_q^*: H^*(H_q(\K_-))\to H^{*-1}(H^{n-q}(\K_{-^c})).
\end{equation}

\begin{thm}\label{maindouble}
    The map $\eta_q$ is an isomorphism whenever $\K$ is tight.      
\end{thm}
\begin{proof}
    The long exact sequences in cohomology of sequences (2) and (3) fit in the following commutative diagram with exact rows for every $l$:
        \[\begin{tikzcd}
{H^l(\ker i_q)} & {H^l(H_q)} & {H^l(\im i_q)} & {H^{l+1}(\ker i_q)} \\
 {H^{l-1}(\ker i_{q-1})} & {H^{l-1}(H^{n-q})} & {H^{l-1}(\im p_q)} & {H^{l-2}(\ker i_{q-1})}
	\arrow[from=1-1, to=1-2]
	\arrow["i_q^*",from=1-2, to=1-3]
	\arrow["\eta_{q}", from=1-2, to=2-2]
	\arrow[from=1-3, to=1-4]
	\arrow["\partial_q^{-1}", from=1-3, to=2-3]
	\arrow["1^*" ,from=2-3, to=2-2]
	\arrow[from=2-4, to=2-3]
    \arrow[from=2-2,to=2-1]
\end{tikzcd}\]
where $\eta$ is the unique map making that diagram commute. If $\K$ is $\F$-tight then $\ker{i_q}=0$ for every $q$, meaning $i_q^*$ and $1^*$ are isomorphisms and so is $\eta_q$.
\end{proof}

 \begin{exm}
     Consider the triangulations of $\mathbb{T}^2$ in Figures 2 and 3.\\

    \begin{minipage}[t]{0.5\textwidth}
    \centering
      \tikzset{every picture/.style={line width=0.75pt}} %set default line width to 0.75pt        

\begin{tikzpicture}[x=0.75pt,y=0.75pt,yscale=-.6,xscale=.6]
%uncomment if require: \path (0,300); %set diagram left start at 0, and has height of 300

%Shape: Rectangle [id:dp6510939799134927] 
\draw  [fill={rgb, 255:red, 135; green, 237; blue, 193 }  ,fill opacity=0.28 ][line width=1.5]  (207.78,51.34) -- (397.72,51.34) -- (397.72,241.28) -- (207.78,241.28) -- cycle ;
%Straight Lines [id:da9789142507615645] 
\draw [line width=1.5]    (207.78,241.28) -- (397.72,51.34) ;
%Straight Lines [id:da22168092660329408] 
\draw [line width=1.5]    (333.99,51.72) -- (333.99,114.66) ;
%Straight Lines [id:da7564414758167872] 
\draw [line width=1.5]    (271.09,179.47) -- (271.09,241.46) ;
%Straight Lines [id:da7478700010436139] 
\draw [line width=1.5]    (333.99,113.9) -- (397.72,113.9) ;
%Straight Lines [id:da7961982139419536] 
\draw [line width=1.5]    (208.15,178.72) -- (271.09,178.72) ;
%Shape: Circle [id:dp21743872903953587] 
\draw  [fill={rgb, 255:red, 80; green, 227; blue, 194 }  ,fill opacity=1 ] (390.93,51.34) .. controls (390.93,47.6) and (393.97,44.56) .. (397.72,44.56) .. controls (401.46,44.56) and (404.5,47.6) .. (404.5,51.34) .. controls (404.5,55.09) and (401.46,58.13) .. (397.72,58.13) .. controls (393.97,58.13) and (390.93,55.09) .. (390.93,51.34) -- cycle ;
%Straight Lines [id:da4197777021022868] 
\draw [line width=1.5]    (208,113.9) -- (271.09,51.91) ;
%Straight Lines [id:da21373371244504968] 
\draw [line width=1.5]    (333.99,241.28) -- (397.72,178.72) ;
%Shape: Circle [id:dp5331023902798833] 
\draw  [fill={rgb, 255:red, 80; green, 227; blue, 194 }  ,fill opacity=1 ] (390.93,113.9) .. controls (390.93,110.16) and (393.97,107.12) .. (397.72,107.12) .. controls (401.46,107.12) and (404.5,110.16) .. (404.5,113.9) .. controls (404.5,117.65) and (401.46,120.69) .. (397.72,120.69) .. controls (393.97,120.69) and (390.93,117.65) .. (390.93,113.9) -- cycle ;
%Shape: Ellipse [id:dp8365061758269366] 
\draw  [fill={rgb, 255:red, 80; green, 227; blue, 194 }  ,fill opacity=1 ] (390.93,241.28) .. controls (390.93,237.53) and (393.97,234.49) .. (397.72,234.49) .. controls (401.46,234.49) and (404.5,237.53) .. (404.5,241.28) .. controls (404.5,245.02) and (401.46,248.06) .. (397.72,248.06) .. controls (393.97,248.06) and (390.93,245.02) .. (390.93,241.28) -- cycle ;
%Shape: Ellipse [id:dp23401530811959304] 
\draw  [fill={rgb, 255:red, 80; green, 227; blue, 194 }  ,fill opacity=1 ] (201,241.28) .. controls (201,237.53) and (204.04,234.49) .. (207.78,234.49) .. controls (211.53,234.49) and (214.57,237.53) .. (214.57,241.28) .. controls (214.57,245.02) and (211.53,248.06) .. (207.78,248.06) .. controls (204.04,248.06) and (201,245.02) .. (201,241.28) -- cycle ;
%Shape: Ellipse [id:dp478253713878772] 
\draw  [fill={rgb, 255:red, 80; green, 227; blue, 194 }  ,fill opacity=1 ] (201.36,178.72) .. controls (201.36,174.97) and (204.4,171.94) .. (208.15,171.94) .. controls (211.89,171.94) and (214.93,174.97) .. (214.93,178.72) .. controls (214.93,182.47) and (211.89,185.5) .. (208.15,185.5) .. controls (204.4,185.5) and (201.36,182.47) .. (201.36,178.72) -- cycle ;
%Shape: Circle [id:dp9873805905330107] 
\draw  [fill={rgb, 255:red, 80; green, 227; blue, 194 }  ,fill opacity=1 ] (201,51.34) .. controls (201,47.6) and (204.04,44.56) .. (207.78,44.56) .. controls (211.53,44.56) and (214.57,47.6) .. (214.57,51.34) .. controls (214.57,55.09) and (211.53,58.13) .. (207.78,58.13) .. controls (204.04,58.13) and (201,55.09) .. (201,51.34) -- cycle ;
%Shape: Ellipse [id:dp7122742550340906] 
\draw  [fill={rgb, 255:red, 80; green, 227; blue, 194 }  ,fill opacity=1 ] (327.21,241.28) .. controls (327.21,237.53) and (330.24,234.49) .. (333.99,234.49) .. controls (337.74,234.49) and (340.77,237.53) .. (340.77,241.28) .. controls (340.77,245.02) and (337.74,248.06) .. (333.99,248.06) .. controls (330.24,248.06) and (327.21,245.02) .. (327.21,241.28) -- cycle ;
%Shape: Ellipse [id:dp10744862191878146] 
\draw  [fill={rgb, 255:red, 80; green, 227; blue, 194 }  ,fill opacity=1 ] (264.31,51.91) .. controls (264.31,48.16) and (267.35,45.12) .. (271.09,45.12) .. controls (274.84,45.12) and (277.88,48.16) .. (277.88,51.91) .. controls (277.88,55.65) and (274.84,58.69) .. (271.09,58.69) .. controls (267.35,58.69) and (264.31,55.65) .. (264.31,51.91) -- cycle ;
%Straight Lines [id:da06020648210014301] 
\draw [line width=1.5]    (208,113.9) -- (271.09,178.72) ;
%Straight Lines [id:da5929203880441459] 
\draw [line width=1.5]    (208,113.9) -- (333.99,51.72) ;
%Straight Lines [id:da2574369288486765] 
\draw [line width=1.5]    (271.2,177.31) -- (333.99,51.72) ;
%Straight Lines [id:da6494712225439029] 
\draw [line width=1.5]    (271.09,241.46) -- (335.91,114.66) ;
%Straight Lines [id:da6957624500039706] 
\draw [line width=1.5]    (335.91,114.66) -- (397.72,178.72) ;
%Straight Lines [id:da7339279110798539] 
\draw [line width=1.5]    (271.09,241.46) -- (397.72,178.72) ;
%Shape: Ellipse [id:dp6860437780966525] 
\draw  [fill={rgb, 255:red, 80; green, 227; blue, 194 }  ,fill opacity=1 ] (264.31,241.46) .. controls (264.31,237.72) and (267.35,234.68) .. (271.09,234.68) .. controls (274.84,234.68) and (277.88,237.72) .. (277.88,241.46) .. controls (277.88,245.21) and (274.84,248.25) .. (271.09,248.25) .. controls (267.35,248.25) and (264.31,245.21) .. (264.31,241.46) -- cycle ;
%Shape: Circle [id:dp5893558539501891] 
\draw  [fill={rgb, 255:red, 80; green, 227; blue, 194 }  ,fill opacity=1 ] (264.31,179.47) .. controls (264.31,175.73) and (267.35,172.69) .. (271.09,172.69) .. controls (274.84,172.69) and (277.88,175.73) .. (277.88,179.47) .. controls (277.88,183.22) and (274.84,186.26) .. (271.09,186.26) .. controls (267.35,186.26) and (264.31,183.22) .. (264.31,179.47) -- cycle ;
%Shape: Ellipse [id:dp24291604741511197] 
\draw  [fill={rgb, 255:red, 80; green, 227; blue, 194 }  ,fill opacity=1 ] (201.22,113.9) .. controls (201.22,110.16) and (204.26,107.12) .. (208,107.12) .. controls (211.75,107.12) and (214.79,110.16) .. (214.79,113.9) .. controls (214.79,117.65) and (211.75,120.69) .. (208,120.69) .. controls (204.26,120.69) and (201.22,117.65) .. (201.22,113.9) -- cycle ;
%Shape: Ellipse [id:dp9502622825109055] 
\draw  [fill={rgb, 255:red, 80; green, 227; blue, 194 }  ,fill opacity=1 ] (329.13,114.66) .. controls (329.13,110.91) and (332.17,107.87) .. (335.91,107.87) .. controls (339.66,107.87) and (342.7,110.91) .. (342.7,114.66) .. controls (342.7,118.4) and (339.66,121.44) .. (335.91,121.44) .. controls (332.17,121.44) and (329.13,118.4) .. (329.13,114.66) -- cycle ;
%Shape: Circle [id:dp22006859303522475] 
\draw  [fill={rgb, 255:red, 80; green, 227; blue, 194 }  ,fill opacity=1 ] (390.93,178.72) .. controls (390.93,174.97) and (393.97,171.94) .. (397.72,171.94) .. controls (401.46,171.94) and (404.5,174.97) .. (404.5,178.72) .. controls (404.5,182.47) and (401.46,185.5) .. (397.72,185.5) .. controls (393.97,185.5) and (390.93,182.47) .. (390.93,178.72) -- cycle ;
%Shape: Ellipse [id:dp7783128215466757] 
\draw  [fill={rgb, 255:red, 80; green, 227; blue, 194 }  ,fill opacity=1 ] (327.21,51.72) .. controls (327.21,47.97) and (330.24,44.94) .. (333.99,44.94) .. controls (337.74,44.94) and (340.77,47.97) .. (340.77,51.72) .. controls (340.77,55.47) and (337.74,58.5) .. (333.99,58.5) .. controls (330.24,58.5) and (327.21,55.47) .. (327.21,51.72) -- cycle ;

% Text Node
\draw (209.78,54.74) node [anchor=north west][inner sep=0.75pt]    {$1$};
% Text Node
\draw (399.72,244.68) node [anchor=north west][inner sep=0.75pt]    {$1$};
% Text Node
\draw (209.78,244.68) node [anchor=north west][inner sep=0.75pt]    {$1$};
% Text Node
\draw (399.72,54.74) node [anchor=north west][inner sep=0.75pt]    {$1$};
% Text Node
\draw (273.09,55.31) node [anchor=north west][inner sep=0.75pt]    {$2$};
% Text Node
\draw (273.09,244.86) node [anchor=north west][inner sep=0.75pt]    {$2$};
% Text Node
\draw (335.99,55.12) node [anchor=north west][inner sep=0.75pt]    {$3$};
% Text Node
\draw (335.99,244.68) node [anchor=north west][inner sep=0.75pt]    {$3$};
% Text Node
\draw (210,124.09) node [anchor=north west][inner sep=0.75pt]    {$4$};
% Text Node
\draw (399.72,117.3) node [anchor=north west][inner sep=0.75pt]    {$4$};
% Text Node
\draw (399.72,182.12) node [anchor=north west][inner sep=0.75pt]    {$5$};
% Text Node
\draw (210.15,188.9) node [anchor=north west][inner sep=0.75pt]    {$5$};
% Text Node
\draw (279.88,182.87) node [anchor=north west][inner sep=0.75pt]    {$7$};
% Text Node
\draw (337.91,124.84) node [anchor=north west][inner sep=0.75pt]    {$6$};

\end{tikzpicture}
      
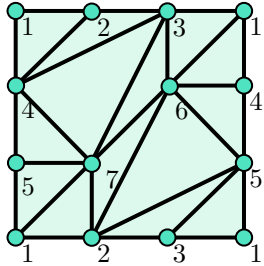
\captionof{figure}{Minimal triangulation $\K^1$ of $\mathbb{T}^2$}
    \end{minipage}\begin{minipage}[t]{0.5\textwidth}
    \centering
    \tikzset{every picture/.style={line width=0.75pt}} %set default line width to 0.75pt        

\begin{tikzpicture}[x=0.75pt,y=0.75pt,yscale=-.6,xscale=.6]
%uncomment if require: \path (0,300); %set diagram left start at 0, and has height of 300

%Shape: Rectangle [id:dp6510939799134927] 
\draw  [fill={rgb, 255:red, 135; green, 237; blue, 193 }  ,fill opacity=0.28 ][line width=1.5]  (207.78,51.34) -- (397.72,51.34) -- (397.72,241.28) -- (207.78,241.28) -- cycle ;
%Straight Lines [id:da9789142507615645] 
\draw [line width=1.5]    (207.78,241.28) -- (397.72,51.34) ;
%Straight Lines [id:da22168092660329408] 
\draw [line width=1.5]    (333.99,51.72) -- (333.99,241.28) ;
%Straight Lines [id:da7564414758167872] 
\draw [line width=1.5]    (271.09,51.91) -- (271.09,241.46) ;
%Straight Lines [id:da7478700010436139] 
\draw [line width=1.5]    (208,113.9) -- (397.72,113.9) ;
%Straight Lines [id:da7961982139419536] 
\draw [line width=1.5]    (208.15,178.72) -- (397.72,178.72) ;
%Shape: Circle [id:dp21743872903953587] 
\draw  [fill={rgb, 255:red, 80; green, 227; blue, 194 }  ,fill opacity=1 ] (390.93,51.34) .. controls (390.93,47.6) and (393.97,44.56) .. (397.72,44.56) .. controls (401.46,44.56) and (404.5,47.6) .. (404.5,51.34) .. controls (404.5,55.09) and (401.46,58.13) .. (397.72,58.13) .. controls (393.97,58.13) and (390.93,55.09) .. (390.93,51.34) -- cycle ;
%Straight Lines [id:da4197777021022868] 
\draw [line width=1.5]    (208,113.9) -- (271.09,51.91) ;
%Straight Lines [id:da0361550252960644] 
\draw [line width=1.5]    (208.15,178.72) -- (333.99,51.72) ;
%Straight Lines [id:da21373371244504968] 
\draw [line width=1.5]    (333.99,241.28) -- (397.72,178.72) ;
%Straight Lines [id:da5969888889692302] 
\draw [line width=1.5]    (271.09,241.46) -- (397.72,113.9) ;
%Shape: Circle [id:dp5331023902798833] 
\draw  [fill={rgb, 255:red, 80; green, 227; blue, 194 }  ,fill opacity=1 ] (390.93,113.9) .. controls (390.93,110.16) and (393.97,107.12) .. (397.72,107.12) .. controls (401.46,107.12) and (404.5,110.16) .. (404.5,113.9) .. controls (404.5,117.65) and (401.46,120.69) .. (397.72,120.69) .. controls (393.97,120.69) and (390.93,117.65) .. (390.93,113.9) -- cycle ;
%Shape: Circle [id:dp22006859303522475] 
\draw  [fill={rgb, 255:red, 80; green, 227; blue, 194 }  ,fill opacity=1 ] (390.93,178.72) .. controls (390.93,174.97) and (393.97,171.94) .. (397.72,171.94) .. controls (401.46,171.94) and (404.5,174.97) .. (404.5,178.72) .. controls (404.5,182.47) and (401.46,185.5) .. (397.72,185.5) .. controls (393.97,185.5) and (390.93,182.47) .. (390.93,178.72) -- cycle ;
%Shape: Ellipse [id:dp8365061758269366] 
\draw  [fill={rgb, 255:red, 80; green, 227; blue, 194 }  ,fill opacity=1 ] (390.93,241.28) .. controls (390.93,237.53) and (393.97,234.49) .. (397.72,234.49) .. controls (401.46,234.49) and (404.5,237.53) .. (404.5,241.28) .. controls (404.5,245.02) and (401.46,248.06) .. (397.72,248.06) .. controls (393.97,248.06) and (390.93,245.02) .. (390.93,241.28) -- cycle ;
%Shape: Ellipse [id:dp7783128215466757] 
\draw  [fill={rgb, 255:red, 80; green, 227; blue, 194 }  ,fill opacity=1 ] (327.21,51.72) .. controls (327.21,47.97) and (330.24,44.94) .. (333.99,44.94) .. controls (337.74,44.94) and (340.77,47.97) .. (340.77,51.72) .. controls (340.77,55.47) and (337.74,58.5) .. (333.99,58.5) .. controls (330.24,58.5) and (327.21,55.47) .. (327.21,51.72) -- cycle ;
%Shape: Circle [id:dp2337836226287241] 
\draw  [fill={rgb, 255:red, 80; green, 227; blue, 194 }  ,fill opacity=1 ] (327.62,177.68) .. controls (327.62,173.94) and (330.66,170.9) .. (334.41,170.9) .. controls (338.15,170.9) and (341.19,173.94) .. (341.19,177.68) .. controls (341.19,181.43) and (338.15,184.47) .. (334.41,184.47) .. controls (330.66,184.47) and (327.62,181.43) .. (327.62,177.68) -- cycle ;
%Shape: Ellipse [id:dp6860437780966525] 
\draw  [fill={rgb, 255:red, 80; green, 227; blue, 194 }  ,fill opacity=1 ] (264.31,241.46) .. controls (264.31,237.72) and (267.35,234.68) .. (271.09,234.68) .. controls (274.84,234.68) and (277.88,237.72) .. (277.88,241.46) .. controls (277.88,245.21) and (274.84,248.25) .. (271.09,248.25) .. controls (267.35,248.25) and (264.31,245.21) .. (264.31,241.46) -- cycle ;
%Shape: Ellipse [id:dp23401530811959304] 
\draw  [fill={rgb, 255:red, 80; green, 227; blue, 194 }  ,fill opacity=1 ] (201,241.28) .. controls (201,237.53) and (204.04,234.49) .. (207.78,234.49) .. controls (211.53,234.49) and (214.57,237.53) .. (214.57,241.28) .. controls (214.57,245.02) and (211.53,248.06) .. (207.78,248.06) .. controls (204.04,248.06) and (201,245.02) .. (201,241.28) -- cycle ;
%Shape: Ellipse [id:dp478253713878772] 
\draw  [fill={rgb, 255:red, 80; green, 227; blue, 194 }  ,fill opacity=1 ] (201.36,178.72) .. controls (201.36,174.97) and (204.4,171.94) .. (208.15,171.94) .. controls (211.89,171.94) and (214.93,174.97) .. (214.93,178.72) .. controls (214.93,182.47) and (211.89,185.5) .. (208.15,185.5) .. controls (204.4,185.5) and (201.36,182.47) .. (201.36,178.72) -- cycle ;
%Shape: Ellipse [id:dp24291604741511197] 
\draw  [fill={rgb, 255:red, 80; green, 227; blue, 194 }  ,fill opacity=1 ] (201.22,113.9) .. controls (201.22,110.16) and (204.26,107.12) .. (208,107.12) .. controls (211.75,107.12) and (214.79,110.16) .. (214.79,113.9) .. controls (214.79,117.65) and (211.75,120.69) .. (208,120.69) .. controls (204.26,120.69) and (201.22,117.65) .. (201.22,113.9) -- cycle ;
%Shape: Circle [id:dp9873805905330107] 
\draw  [fill={rgb, 255:red, 80; green, 227; blue, 194 }  ,fill opacity=1 ] (201,51.34) .. controls (201,47.6) and (204.04,44.56) .. (207.78,44.56) .. controls (211.53,44.56) and (214.57,47.6) .. (214.57,51.34) .. controls (214.57,55.09) and (211.53,58.13) .. (207.78,58.13) .. controls (204.04,58.13) and (201,55.09) .. (201,51.34) -- cycle ;
%Shape: Circle [id:dp7753962298947442] 
\draw  [fill={rgb, 255:red, 80; green, 227; blue, 194 }  ,fill opacity=1 ] (264.29,115.22) .. controls (264.29,111.47) and (267.32,108.44) .. (271.07,108.44) .. controls (274.81,108.44) and (277.85,111.47) .. (277.85,115.22) .. controls (277.85,118.97) and (274.81,122) .. (271.07,122) .. controls (267.32,122) and (264.29,118.97) .. (264.29,115.22) -- cycle ;
%Shape: Ellipse [id:dp9502622825109055] 
\draw  [fill={rgb, 255:red, 80; green, 227; blue, 194 }  ,fill opacity=1 ] (329.13,114.66) .. controls (329.13,110.91) and (332.17,107.87) .. (335.91,107.87) .. controls (339.66,107.87) and (342.7,110.91) .. (342.7,114.66) .. controls (342.7,118.4) and (339.66,121.44) .. (335.91,121.44) .. controls (332.17,121.44) and (329.13,118.4) .. (329.13,114.66) -- cycle ;
%Shape: Ellipse [id:dp7122742550340906] 
\draw  [fill={rgb, 255:red, 80; green, 227; blue, 194 }  ,fill opacity=1 ] (327.21,241.28) .. controls (327.21,237.53) and (330.24,234.49) .. (333.99,234.49) .. controls (337.74,234.49) and (340.77,237.53) .. (340.77,241.28) .. controls (340.77,245.02) and (337.74,248.06) .. (333.99,248.06) .. controls (330.24,248.06) and (327.21,245.02) .. (327.21,241.28) -- cycle ;
%Shape: Ellipse [id:dp10744862191878146] 
\draw  [fill={rgb, 255:red, 80; green, 227; blue, 194 }  ,fill opacity=1 ] (264.31,51.91) .. controls (264.31,48.16) and (267.35,45.12) .. (271.09,45.12) .. controls (274.84,45.12) and (277.88,48.16) .. (277.88,51.91) .. controls (277.88,55.65) and (274.84,58.69) .. (271.09,58.69) .. controls (267.35,58.69) and (264.31,55.65) .. (264.31,51.91) -- cycle ;
%Shape: Circle [id:dp5893558539501891] 
\draw  [fill={rgb, 255:red, 80; green, 227; blue, 194 }  ,fill opacity=1 ] (264.31,179.47) .. controls (264.31,175.73) and (267.35,172.69) .. (271.09,172.69) .. controls (274.84,172.69) and (277.88,175.73) .. (277.88,179.47) .. controls (277.88,183.22) and (274.84,186.26) .. (271.09,186.26) .. controls (267.35,186.26) and (264.31,183.22) .. (264.31,179.47) -- cycle ;

% Text Node
\draw (209.78,54.74) node [anchor=north west][inner sep=0.75pt]    {$1$};
% Text Node
\draw (399.72,244.68) node [anchor=north west][inner sep=0.75pt]    {$1$};
% Text Node
\draw (209.78,244.68) node [anchor=north west][inner sep=0.75pt]    {$1$};
% Text Node
\draw (399.72,54.74) node [anchor=north west][inner sep=0.75pt]    {$1$};
% Text Node
\draw (273.09,55.31) node [anchor=north west][inner sep=0.75pt]    {$2$};
% Text Node
\draw (260,248.4) node [anchor=north west][inner sep=0.75pt]    {$2$};
% Text Node
\draw (335.99,55.12) node [anchor=north west][inner sep=0.75pt]    {$3$};
% Text Node
\draw (335.99,244.68) node [anchor=north west][inner sep=0.75pt]    {$3$};
% Text Node
\draw (210,117.3) node [anchor=north west][inner sep=0.75pt]    {$4$};
% Text Node
\draw (399.72,117.3) node [anchor=north west][inner sep=0.75pt]    {$4$};
% Text Node
\draw (399.72,182.12) node [anchor=north west][inner sep=0.75pt]    {$5$};
% Text Node
\draw (210.15,182.12) node [anchor=north west][inner sep=0.75pt]    {$5$};
% Text Node
\draw (273.09,182.87) node [anchor=north west][inner sep=0.75pt]    {$8$};
% Text Node
\draw (336.41,181.08) node [anchor=north west][inner sep=0.75pt]    {$9$};
% Text Node
\draw (337.91,118.06) node [anchor=north west][inner sep=0.75pt]    {$7$};
% Text Node
\draw (273.07,118.62) node [anchor=north west][inner sep=0.75pt]    {$6$};

\end{tikzpicture}
    
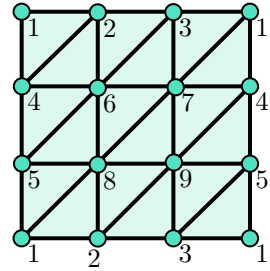
\captionof{figure}{Standard triangulation $\K^2$ of $\mathbb{T}^2$}.
     \end{minipage}\\
     
    With computational aid (see \cite{script}) and the main result of \cite{ruiz2025doubleuberposethomology} we can reliably compute $H^*(H_*(\K_-))$ for any complex $\K$, in particular, for these examples we get that:
      
          \[H^l(H_q(\K^1_-;\Z/2))=\left\{\begin{array}{cc}
             \Z/2  & \text{for }(q,l)=(0,1),(1,3),(1,5),(2,7) \\
             0  & \text{otherwise,} 
          \end{array}\right.\]
          
          \[H^l(H_q(\K^2_-;\Z/2))=\left\{\begin{array}{cc}
             (\Z/2)^2  & \text{for }(q,l)=(0,2)\text{ or }(1,4)\\
             \Z/2  &  \text{for }(q,l)=(1,7)\text{ or }(2,9)\\
             0 & \text{otherwise.}
          \end{array}\right.\]
    \end{exm}

    The duality from Theorem \ref{maindouble} in $H^*(H_*(\K^1_-))$ is evident, as in this example, the bidgree $(0,1)$ is dual to $(2,7)$, and $(1,3)$ is dual to $(1,5)$.
\bibliographystyle{alpha}
\bibliography{ref.bib}
\end{document}